\documentclass{amsart}
\usepackage[english]{babel}

\usepackage{graphicx} 
\usepackage{longtable} 
\usepackage[toc,page]{appendix}
\usepackage[all]{xy}
\usepackage{hyperref}
\usepackage{fancyhdr}
\usepackage{amscd}
\usepackage{amssymb}
\usepackage{pb-diagram}
\usepackage{mathrsfs}

\newcommand{\pic}{\underline{\mbox{Pic}}}

\newcommand{\piclog}{\underline{\mbox{Pic}}^{log}}
\newcommand{\piclogtau}{\underline{\mbox{Pic}}^{log,[0]}}

\newcommand{\piclogdag}{\underline{\mbox{Pic}}^{log, \dag}}

\newcommand{\ox}{\mathcal{O}}

\newcommand{\frd}{\rightarrow}

\newcommand{\gm}{{\mathbb{G}_m}}

\newcommand{\nn}{\mathbb{N}}

\newcommand{\zz}{\mathbb{Z}}

\newtheorem{thm}{Theorem}[section]

\newtheorem{lm}[thm]{Lemma}

\newtheorem{cor}[thm]{Corollary}

\theoremstyle{definition}

\newtheorem{definition}[thm]{Definition}
\newtheorem{rmk}[thm]{Remark}

\numberwithin{equation}{section}

\begin{document}
\title{On the Log-Picard functor for aligned degenerations of curves}
\author{Alberto Bellardini}
\address{KU Leuven University, Department Wiskunde\\
Celestijnenlaan 200B\\
3001 Heverlee\\
Belgium}
\email{albertobellardini@yahoo.it}
\begin{abstract}
We show that a particular subfunctor of the relative logarithmic Picard functor for families of aligned, log semistable curves over a regular base scheme and smooth over an open dense subscheme of the base is representable by a smooth algebraic space which is slightly more separated than the classical relative Picard functor. We show the existence of its maximal separated quotient and give a new functorial interpretation in terms of logarithmic geometry for the N\'eron model of the relative Picard functor of the smooth locus when we restrict to transversal pull backs from spectra of discrete valuation rings. 
We also show that aligned degenerations are log cohomologically flat.
\end{abstract}
 \maketitle

\tableofcontents

\section*{Introduction}

\noindent In this work we study the representability of the \emph{log Picard functor} associated to families of log semistable curves (defined in \ref{logsemistcurves}). For aligned, log semistable degenerations of smooth curves we relate this functor with the N\'eron model of the associated relative Picard functor of the smooth locus. In the literature the notion of N\'eron models it is usually considered in the case where the base scheme is one dimensional. For us a N\'eron model is defined as follows. 
\begin{definition}\label{neronalgebr}
Let $S$ be a scheme, $U$ be a schematically dense open subset of $S$ and $X$ be an algebraic space over $U$ which is smooth over $U$. A N\'eron model for $X$ over $S$ is a smooth and separated algebraic space $\mathcal{N}/S$ together with an
isomorphism $X \cong \mathcal{N}_U$ satisfying the following universal property: let $T \frd S$ be a
smooth morphism of algebraic spaces and $f : T_U \frd X$ be any $U$-morphism then there exists a unique $S$-morphism $F : T \frd \mathcal{N}$ such that $F|_U = f$.
\end{definition}
In the definition we don't require that the N\'eron model is of finite type.
By a recent work of D. Holmes (\cite{hol}) it is known that given a family of nodal curves over a regular base which is smooth over a schematically dense open subscheme of the base and with regular total space then the jacobian of the curve on the smooth locus admits a N\'eron model if and only if the labelled graphs of the fibers satisfy a combinatorial condition called \emph{alignment} (see definition \ref{defaligned} and theorem \ref{hol:main}).\\
As far as we know the logarithmic Picard functor this was firstly introduced by Kajiwara in the paper \cite{kaj} for log curves without self intersection over a field. 
A relative version of this functor for semistable families with fibers of any dimension was studied by M.C. Olsson in \cite{olpic} where a representability result is proved under a condition which we call in definition \ref{logcohfl} \emph{log cohomological flatness}.
It is not clear to us under which general assumptions log cohomological flatness is satisfied for families having fibers of dimension at least 2. 
In theorem \ref{teorema} we show that for semistable families of curves the notion of alignment implies log cohomological flatness.
We use this to give a logarithmic interpretation for the N\'eron model of the relative Picard functor of the smooth locus. 
To be more precise given a morphism of log schemes $f:(X,M_X)\frd (S,M_S)$ we define
$$
\piclogtau_{X/S}
$$
to be the \'etale sheafification on the big \'etale site of $S$ of the functor of isomorphism classes of $M_X^{gp}$-torsors such that the contracted product with the sheaf $\overline{M}_X^{gp}$ is isomorphic to the trivial $\overline{M}_X^{gp}$-torsor \'etale locally on $S$.
If $S$ is the spectrum of a discrete valuation ring and $f:X\frd S$ is a morphism of schemes we have a canonical log structure induced from the special fiber. In particular we get in a canonical way a morphism of log schemes
$$
f^{\dag}:(X,M_X^{\dag})\frd (S,M_S^{\dag})
$$
where $M_X^{\dag}$ and $M_S^{\dag}$ are the special fiber log structures. When $S$ is the spectrum of a discrete valuation ring we will always denote with $(-)^{\dag}$ the log structure induced from the special fiber.
We define 
$$\piclogdag_{X/S}$$
as the \'etale sheafification of the functor classifying isomorphism classes of $M_X^{\dag,gp}$-torsors.
Having introduced these notations (see also \ref{logpicardstack} and \ref{maxsepquot} for more details) we can state the main result of this paper.
\begin{thm}[Theorem \ref{teoremamain}, corollary \ref{functnerodvr}]
Let $(S,M_S)$ be a log scheme such that $S$ is regular. Assume that the log structure $M_S$ is trivial on a schematically dense open $U\subset S$.
Let 
$$
f:(C,M_C)\frd (S,M_S)
$$
\noindent be a proper, special, log semistable curve over $S$ (see definition in \ref{specialmorphism}), strict over $U$. 
If the underlying family of schemes $C/S$ is aligned then
\begin{enumerate}
    \item the maximal separated quotient $Q^{log}$ of $\piclogtau_{C/S}$ 
        is an algebraic space and if $C$ is regular this is a N\'eron model for $\pic_{C_U/U}$ in the sense of definition \ref{neronalgebr};
    \item the algebraic space $Q^{log}$ has the property that for any discrete valuation ring $V$ and morphism $T=Spec(V)\frd S$ mapping the generic point to the open $U$ there is a canonical isomorphism of group functors

$$
Q_T^{log}\cong \piclogdag_{C_T/T}
$$
\end{enumerate}

Furthermore given a proper nodal curve over the spectrum of a discrete valuation ring $T$ with smooth generic fiber $C_{\eta}$ then the sheaf $\piclogdag_{C/T}$ is representable by an algebraic group which is smooth and separated over $T$. If $C$ is regular then $\piclogdag_{C/T}$ is the N\'eron model of $\pic_{C_{\eta}/\eta}$.
\end{thm}
Over discrete valuation rings, where alignment is automatic, there is a naive way to understand the fact that the maximal separated quotient of the relative Picard functor is at least a subgroup of the relative logarithmic Picard functor computed with the special fiber log structure.
It is known that for the relative Picard functor over discrete valuation rings the obstruction to being separated comes from line bundles arising from divisors concentrated in the special fiber. If one considers the log structure coming from the divisor of the special fiber, one sees that the transition functions describing the \v{C}ech 1-cocycles of these particular line bundles are naturally sections of the sheaf in groups which is the groupification of the sheaf of the log structure. 
In particular they can be interpreted as different trivializations of the trivial torsor under this group.
If we consider isomorphism classes of torsors under this group, these line bundles define the same isomorphism class, namely the one of the trivial torsor.
In other words these line bundles become trivial as logarithmic torsors.
This naive idea can be made more precise and natural by using cohomological methods.\\ 
We don't know how to control the representability using the natural generalization of the notion of the ``special fiber log structure'' to families where the base has dimension greater than one. 
The notion of special log structure introduced by Olsson (which \emph{does not} coincide with the log structure induced by the special fiber for families over discrete valuation rings) produces a better control instead.
Anyway these two log structures have a natural comparison map for families over discrete valuation rings and this can be used to prove part the previous theorem.
Unfortunately we are not able to find non aligned log cohomologically flat examples.

Beside the fact that this functor is closely related to the N\'eron model of the jacobian of these curves, we want to give another motivation about the reason for which we care about it. In \cite{b14} we studied compactifications for jacobians of singular curves having at worst nodal singularities by relating the construction given by Oda and Seshadri in \cite{os} with the Relatively Complete Models of D. Mumford, G. Faltings and C.-L. Chai as presented in \cite{fc}. In particular we were able to recover and uniformize, in a modular way, some coarse moduli spaces of T. Oda and C.S. Seshadri, without using geometric invariant theory, and we have a functor for the uniformizing object.
Here the word ``some'' means that we can do this only for particular choices of the polarization one uses to construct the compactified jacobians of T. Oda and C.S. Seshadri. 
The sheaves we used to uniformize the compactified jacobians naturally correspond to certain logarithmic torsors and we used the formalism of log geometry to give functoriality to our construction.
In particular we showed that the sheaves we obtained have a natural interpretation in terms of the logarithmic Picard functor of a certain analytic covering of the curves. In this paper we generalize the representability result 5.0.5 given in \cite{b14} for curves over discrete valuation rings.

This paper is structured as follows. In the first section we recall facts that we need from logarithmic geometry and the notion of aligned curves.
In section 2.1 we characterize log cohomological flatness in terms of \'etaleness of the natural morphism from the relative Picard functor to the relative log Picard functor (theorem \ref{generalrepres}). 
Section 2.2 contains the main results of the paper. 
Namely in theorem \ref{teorema} it is stated that speciality and alignment imply log cohomological flatness and the representability $\piclogtau_{C/S}$ by a smooth algebraic space.
Finally theorem \ref{teoremamain} relates this functor with the N\'eron model of the relative Picard functor of the generic fiber.
\section*{Acknowledgments}
We would like to thank David Holmes, Johannes Nicaise, Martin Olsson and Filippo Viviani for many useful conversations. This work was supported by the European Research Council (ERC) Starting Grant MOTZETA (project 306610).

\section{Preliminaries in log geometry}
\subsection{Log semistable curves}\label{lageometrialogaritmica}

Standard facts about log geometry can be found in \cite{ka}, we recall here some definitions from the paper \cite{olpic} that we need in this work.
\begin{definition}
    Given a separably closed field $k$, a scheme $X$ over $k$ is called a \emph{semistable variety} if for any geometric point $\bar{x}\in X$ there exists an \'etale neighborhood $(U,u)$ and positive integers $0\leq m\leq n$ such that 
$U$ is \'etale over 
$$
Spec(k[X_1,\dots,X_n]/(X_1\cdot \dots \cdot X_m))
$$
and the point $u$ is sent to the point corresponding to the ideal $(X_1,\dots,X_n)$.
\end{definition}
We want to given a logarithmic version of this notion.
Given a log structure $M_X$ on a scheme $X$ we have always an exact sequence of sheaves in monoids
$$
0\frd \ox_X^{\times}\frd M_X\frd \overline{M}_X\frd 0
$$
The sheaf $\overline{M}_X$ is usually called the \emph{characteristic of the log structure} $M_X$.

\begin{definition}\label{semistablemorphism}
    A log smooth morphism $f:(X,M_X)\frd (S,M_S)$ is called \emph{ essentially semistable} if for each geometric point $\bar{x}\frd X$ the monoids $(f^{-1}\overline{M}_S)_{\bar{x}}$ and $\overline{M}_{X,\bar{x}}$ are free and there exist isomorphisms $(f^{-1}\overline{M}_S)_{\bar{x}}\cong\nn^r$ and $\overline{M}_{X,\bar{x}}\cong\nn^{r+s}$ such that the induced map
$$
(f^{-1}\overline{M}_S)_{\bar{x}}\frd\overline{M}_{X,\bar{x}}
$$
\noindent is ``multidiagonal'', i.e. one can choose coordinates such that in these coordinates the morphism is given by
$$
1_i\frd \left\{
\begin{array}{cc}
1_i & \mbox{ if }i\neq r\\
1_{r+1}+\dots +1_s & \mbox{ if }i= r\\
\end{array}
\right.
$$
where $1_i$ denotes the vector with zero everywhere and $1$ in the i-th entry.
\end{definition}
\begin{rmk}\label{esseflat}
    Essentially semistable morphism are automatically flat and vertical (\cite[2.3]{olun}). Vertical means that the cokernel of the map
$$
f^{*}M_S\frd M_X
$$
\noindent is a sheaf of groups.
\end{rmk}
Let $f:(X,M_X)\frd (S,M_S)$ be an essentially semistable morphism. For any geometric point $\bar{s}\in S$ let $I(\overline{M}_{S,\bar{s}})$ be the set of irreducible elements in $\overline{M}_{S,\bar{s}}$. Define 
$$
Conn(X_{\bar{s}}):=\left\{\mbox{ connected components of the singular locus of }X_{\bar{s}}\right\}.
$$
Since $f$ is essentially semistable then there is a morphism   
$$
s_{X_{\bar{s}}}:Conn(X_{\bar{s}})\frd I(\overline{M}_{S,\bar{s}})
$$
given by sending a component to the unique irreducible element whose image in $\overline{M}_{X,x}$ is not irreducible, where $x$ is a point in the chosen component.

\begin{definition}\label{specialmorphism}
An essentially semistable morphism of log schemes $$f:(X,M_X)\frd (S,M_S)$$ is called \emph{special} at a geometric point $\bar{s}\in S$ if the previous map
$$
s_{X_{\bar{s}}}:Conn(X_{\bar{s}})\frd I(\overline{M}_{S,\bar{s}})
$$
induces a bijection between the set of connected components of the singular locus of $X_{\bar{s}}$ and $I(\overline{M}_{S,\bar{s}})$. A morphism is called special if it is special at every geometric point of $S$.
\end{definition}
General facts about special morphisms are given in \cite{olun}. We want to describe more precisely the fibers of special morphisms.
Assume for simplicity that the scheme $S$ is the spectrum of a field and that 
$$
f:(X,M_X)\frd (S,M_S)
$$
is special. 
There is an isomorphism, induced by $s_X$,
$$
\overline{M}_S\cong\nn^{Conn(X)}
$$
Given $c\in Conn(X)$ one defines the subsheaves of ``the branches at $c$''
$$
\overline{M}_X\supset\overline{M}_c:=\left\{
\begin{array}{c}
    \mbox{ sections }m\mbox{ of } \overline{M}_X \mbox{ s.t. \'etale locally }\exists \;\mbox{ a section }n\mbox{ of } \overline{M}_X \\
    \mbox{ with }m+n \mbox{ is a multiple of }s_X(c)
\end{array}
\right\}
$$
The preimage in $M_X$ of these sheaves gives subsheaves $M_c$ of $M_X$. One recovers the log structure using the previous sheaves by push-out over $\ox_{X}^{\times}$, i.e. there is an isomorphism
$$
M_X\cong \bigoplus_{c\in Conn(X), \ox_X^{\times}} M_c
$$
 There is also another way to see these sheaves. A connected component corresponding to a $c\in Conn(X)$ is set theoretically defined \'etale locally around a point $x$ by the closed given by an ideal of the form
 \begin{equation}\label{jeic}
J_c=(x_1\cdots \hat{x}_j \cdots x_r)_{j=1}^r
\end{equation}
where $\hat{x}_j$ means that we substitute the function $x_j$ with $1$ in the product. One considers the blowup of $X$ along $J_c$ 
$$
\nu_c:\tilde{X}_c\frd X
$$
and shows (\cite[2.15]{olpic}) that there is an isomorphism 
$$
\overline{M}_{c}\cong \nu_{c,*}\nn
$$
Locally if $\bar{x}$ is a closed point of $X$ with preimage $x_1,\dots,x_r$ in $\tilde{X}_c$ then there is an isomorphism 
$$
\overline{M}_{c,\bar{x}}\stackrel{\sim}{\frd} \bigoplus_{x_i}\nn_{x_i}
$$
\begin{definition}\label{logsemistcurves}
We call a log smooth, proper, integral and vertical morphism of log schemes 
$$
f:(X,M_X)\frd (S,M_S)
$$
\noindent whose geometric fibers are connected, one dimensional and semistable a \emph{log semistable curve}.
\end{definition}
\begin{rmk}
    Observe that a log semistable curve is flat and cohomologically flat in dimension zero by remark \ref{autocohfla}.
\end{rmk}
\begin{rmk}\label{makespecial}
Given a log smooth, proper, integral and vertical morphism 
$$
f:(X,M_X)\frd (S,M_S)
$$
\noindent such that the geometric fibers are semistable it is always possible to change the log structure, i.e. find $M_X^{\sharp}$ on $X$ and $M_S^{\sharp}$ on $S$, in an unique way (up to isomorphism) such that the induced morphism 
$$
f^{\sharp}:(X,M_X^{\sharp})\frd (S,M_S^{\sharp})
$$
is special (\cite[2.6]{olun}). More precisely we can find a canonical cartesian diagram
$$
\xymatrix{
    {(X,M_X)}\ar[d]_{f}\ar[r] & {(X,M_X^{\sharp})}\ar[d]^{f^{\sharp}}\\
    {(S,M_S)}\ar[r] & {(S,M_S^{\sharp})}
}
$$
with the right vertical morphism special and the horizontal arrows inducing isomorphisms on the underlying schemes.
Hence it is harmless to assume speciality if we are only interested in the geometry of the underlying scheme. 
\end{rmk}
\subsection{Aligned curves}
In the paper \cite{hol} D. Holmes introduces a class of semistable curves for which the N\'eron model of the jacobian of the smooth locus is representable by an algebraic space.
He calls these curves \emph{aligned}. Let us recall this construction.
Consider $S$ a regular scheme and $C\frd S$ be a semistable curve over $S$. For any geometric point $\bar{s}\in S$, the fiber $C_{\bar{s}}$ is a semistable curve. Let $\Gamma_{\bar{s}}$ be the dual intersection graph of $C_{\bar{s}}$. 
For any edge $e$ of $\Gamma_{\bar{s}}$, corresponding to a node $c\in C_{\bar{s}}$ we can find an element $\alpha\in\mathfrak{m}_{\ox_{S,\bar{s}}^{et}}$ such that we have an isomorphism
\begin{equation}\label{alpha:element}
\widehat{\ox_{C,c}}\cong \widehat{\ox_{S,\bar{s}}}[[x,y]]/(xy-\alpha).
\end{equation}
The element $\alpha$ is not unique but the ideal
$$
\alpha\ox_{S,s}^{et}
$$
\noindent is unique. 
In particular for any edge $e\in \Gamma_{\bar{s}}$ we get a well defined element in the monoid 
$$
L=\ox_{S,\bar{s}}^{et}/\ox_{S,\bar{s}}^{et,\times}
$$
This gives us a so called $L$-edge labelling, namely a map
$$
l:E\frd L
$$
 where $E$ is the set of edges of $\Gamma_{\bar{s}}$. \\
\begin{definition}[\cite{hol}]\label{defaligned}
Let $M$ be a commutative monoid and $(\Gamma,l,M)$ be a graph with an $M$ edge labelling. We say that $(\Gamma,l,M)$ is \emph{aligned} if for any circuit $H$ in $\Gamma$ and any pairs of edges $e_1,e_2$ in $H$ there exist positive integers $n_1,n_2$ such that
$$
n_1l(e_1)=n_2l(e_2)
$$
Given $S$ a regular scheme, $C\frd S$ a family of semistable curves and a geometric point $\bar{s}\in S$ we say that $C$ is \emph{aligned} at $\bar{s}$ if the triple $(\Gamma_{\bar{s}},l,\ox_{S,\bar{s}}^{et}/\ox_{S,\bar{s}}^{et,\times})$ defined before is aligned. \\
We say that a family of semistable curves $C\frd S$ is \emph{aligned} if it is aligned at any geometric point $\bar{s}\in S$.\\
Given a log semistable curve $(C,M_C)\frd (S,M_S)$ we call it \emph{aligned} if the underlying morphism of schemes $C\frd S$ is aligned.
\end{definition}
%

\begin{rmk}\label{alignmentcod2}

When $S$ is the spectrum of a discrete valuation ring and the curve is generically smooth then the alignment condition is automatically satisfied.
\end{rmk}

The notion of alignment is important in the study of the N\'eron model for families of curves over regular base with dimension bigger than one.
We recall now the following important facts.
\begin{thm}[\cite{hol}]\label{hol:main}
    Let $S$ be a regular scheme and $C\frd S$ be a proper family of semistable curves over $S$. Let $U\subset S$ be a schematically dense open subscheme and assume that the curve $C_U$ obtained by base changing to $U$ is smooth. Let $\pic_{C/S}^{[0]}$ be the closure of $\pic_{C_U/U}^0$ in $\pic_{C/S}$ and $clo(e_U)$ be the closure of the identity section $e_U\in \pic_{C_U/U}^{0}$ in $\pic_{C/S}^{[0]}$. 
\begin{enumerate}

    \item The following are equivalent:
        \begin{itemize}
            \item[(i)] the algebraic space $clo(e_U)$ is \'etale over $S$;
            \item[(ii)] the algebraic space $clo(e_U)$ is flat over $S$;
            \item[(iii)] the family of curves $C\frd S$ is aligned.

        \end{itemize}
    \item If the jacobian $\pic_{C_U/U}^0$ admits a N\'eron model over $S$ then $C\frd S$ is aligned.
    \item If $C$ is regular and the family is aligned then the jacobian $\pic_{C_U/U}^0$ admits a N\'eron model over $S$ and the N\'eron model is isomorphic to the quotient $\pic_{C/S}^{[0]}/clo(e_U)$.
\end{enumerate}
\end{thm}
\begin{definition}\label{closureidentity}
  Define $E$ to be the schematic closure in $\pic_{C/S}$ of the unity section in $\pic_{C_U/U}$.
\end{definition}
Since in our case $\pic_{C/S}$ is representable by an algebraic space in groups locally of finite type we have that $E$ is also representable by an algebraic space in groups locally of finite type.
From the inclusions $E\subset \pic_{C/S}^{[0]}\subset \pic_{C/S}$ we have that
$$
E=clo(e_U)
$$
\begin{cor}\label{Eetale}
    Under the alignment hypothesis of the previous theorem the algebraic group space $E$ is \'etale over $S$. Furthermore in this case $\pic_{C/S}/E$ is a N\'eron model of $\pic_{C_U/U}$ if $C$ is regular.
\end{cor}
\begin{proof}
    Once we know that $E$ is \'etale over $S$, that $C$ is regular and that $\pic_{C/S}$ is smooth then the proof that $\pic_{C/S}/E$ is a N\'eron model is the same argument given in \cite{hol} and we don't repeat it here.
\end{proof}
\section{The Log Picard functor}\label{chlogpic}
First we need some definitions and facts.
\begin{definition}\label{cohomofla}
    Let $f:X\frd S$ be a proper morphism of finite presentation. We call the morphism $f$ \textbf{cohomologically flat in dimension zero} over $S$ if the formation of $f_{*}\ox_X$ commutes with every base change over $S$.
\end{definition}
\begin{rmk}\label{autocohfla}
In the case of flat families of proper curves this happens if for example $S$ is reduced and the dimension of the vector spaces $H^1(X_s,\ox_{X_s})$ is constant.
For families of semistable curves as in our situation we are in a better situation, namely the cohomological flatness in dimension zero follows because the families are flat and with reduced geometric fibers (\cite[III, 7.8.6]{ega}).
\end{rmk}

Having fixed a log scheme $(S,M_S)$, we consider the category $(Sch/S)$ inside the category of log schemes as follows: let $g:T\frd S$ be a morphism then we put on $T$ the log structure $g^{*}M_S$.
\begin{rmk}\label{univsat}
Observe that given a strict morphism $(T,g^{*}M_S)\frd (S,M_S)$ and a fs log scheme $(X,M_X)$ over $(S,M_S)$ then the fiber product $(X_T,M_{X_T})$ in the category of fine and saturated log schemes coincides with the fiber product in the category of log schemes. 
\end{rmk}
In order to distinguish from the classical cohomological flatness we recall the following notion which unfortunately is also called cohomological flatness in dimension zero in \cite{olpic}.
\begin{definition}[\cite{olpic} 4.3]\label{logcohfl}
A morphism $(X,M_X)\frd (S,M_S)$ of log schemes is called \textbf{log cohomologically flat} if for any affine nilpotent closed immersion 
$$Spec(A_0)\frd Spec(A)
$$ over $S$ defined by a square zero ideal, the natural map
$$
H^0(X_A,M_{X_A}^{gp})\frd H^0(X_{A_0},M_{X_{A_0}}^{gp})
$$
is surjective.
\end{definition}
\begin{definition}\label{logpicardstack}

\begin{enumerate}

\item Given a morphism of log schemes 
    $$
    f:(X,M_X)\frd (S,M_S)
    $$
    \noindent the \textbf{log Picard stack} is the stack on $Sch/S$ with respect to the \'etale topology corresponding to the groupoid whose fiber over a scheme $g: T\frd S$ is defined by
$$
\mathcal{P}ic_{X/S}^{log}(T):=\{M_{X_T}^{gp} \mbox{-torsors on } (X_{T})_{et}\}
$$
\item the \textbf{log Picard functor}, denoted with $\piclog_{X/S}$, is the \'etale sheafification on $(Sch/S)_{{et}}$ of the functor of isomorphism classes for the log Picard stack. 
\item let $\mathcal{P}ic_{X/S}^{log,[0]}$ be the substack whose fiber on a scheme $T$ over $S$ is the groupoid of $M_{X_T}^{gp}$-torsors $P$ such that $P\times^{M_{X_T}^{gp}}\overline{M}_{X_T}^{gp}$ is \'etale locally on $T$ trivial. 
\item let $\piclogtau_{X/S}$ be the \'etale sheafification of isomorphism classes in $\mathcal{P}ic_{X/S}^{log,[0]}$.
\end{enumerate}
\end{definition}

We have at our disposal the following representability result.
\begin{thm}[\cite{olpic} 4.6,\cite{olspriv}]\label{olssonrepres}
    Let $f:(X,M_X)\frd (T,M_T)$ be a proper, log smooth and special morphism with semistable geometric fibers. If the morphism is also log cohomologically flat then $\piclogtau_{X/S}$ is representable by an algebraic space.
\end{thm}
\begin{proof}
There is a small mistake in \cite{olpic} concerning the algebraization statement. Namely the last part of lemma 4.18 is not true in general. 
M. Olsson informed me that the proof of \cite[4.6]{olpic} works for $\piclogtau_{X/S}$ (we don't need lemma 4.18 in this case).
\end{proof}
\subsection{Preliminary observations}
Let us fix a base scheme log $(S,M_S)$ and a morphism of log schemes $f:(C,M_C)\frd (S,M_S)$. Consider the exact sequence of sheaves on $C$ for the \'etale topology
$$
0\frd\ox_C^{\times}\frd M_C^{gp}\frd \overline{M}_C^{gp}\frd 0
$$
This gives us a long exact sequence
\begin{equation}\label{longexactseq}
   \dots \frd f_{*}\overline{M}_{C}^{gp}\stackrel{\delta}{\frd} R^1f_{*}\ox_C^{\times}\frd R^1f_{*}M_{C}^{gp}\frd R^{1}f_{*}\overline{M}_{C}^{gp}\frd\dots \end{equation}
\noindent of sheaves in groups. This sequence is functorial with respect to base change $T\frd S$ and in this way we obtain a sequence on the big \'etale site of $S$.
Our approach is to study the representability of the log Picard functor by analyzing the representability of the elements in this complex.
We start now with a general theorem in the special case which characterizes the log cohomological flatness.
\begin{thm}\label{generalrepres}
    Let $f:(C,M_C)\rightarrow (S,M_S)$ be a special, log semistable curve with $S$ excellent. Let $K$ be the kernel of the morphism $\pic_{C/S}\rightarrow \piclog_{C/S}$ induced by the sequence \ref{longexactseq}. The following are equivalent:
            \begin{enumerate}
                \item $K$ is representable by an algebraic space \'etale over $S$;
                \item the morphism $f:(C,M_C)\frd (S,M_S)$ is log cohomologically flat.
            \end{enumerate}
If one of the previous condition is satisfied then $\piclogtau_{C/S}$ is representable by an algebraic space smooth over $S$.
\end{thm}
\begin{proof}
    The last representability statement follows from theorem \ref{olssonrepres}. Smoothness is discussed later.
Let us first prove the following lemma which also works for families with fibers of dimension bigger than one.
\begin{lm}\label{mbaralgsp}
  Let $(X,M_X)\frd (S,M_S)$ be a proper, log smooth, special morphism with semistable geometric fibers and assume that $S$ is excellent. 
The \'etale sheaves 
$
R^{1}f_{*}\overline{M}_{X}^{gp}
$
 and $f_{*}\overline{M}_{X}^{gp}$ considered as sheaves on $(Sch/S)_{et}$ are representable by algebraic spaces in groups which are locally separated and \'etale over $S$.
 \end{lm}

\begin{proof}
   First we are going to verify the conditions of Artin's representability theorem in the form of theorem 5.3 in \cite{aalg} for the sheaves $f_{*}\overline{M}_{X}^{gp}$ and $R^{1}f_{*}\overline{M}_{X}^{gp}$.
Observe that in loc. cit. Artin restricts to the case in which the scheme $S$ is of finite type over a field or over an excellent Dedekind domain in order to obtain the algebraization property.
It has been shown in \cite{cdej} that this condition can be omitted. We can apply Artin theorem also in the case that $S$ is assumed to be only excellent.

Condition $[0^{\prime}]$ follows because there is no difference between \'etale and flat log structure (\cite[A.1]{ollog}) hence we have that $f_{*}\overline{M}_{X}^{gp}$ and $R^1f_{*}\overline{M}_{X}^{gp}$ are sheaves in the fppf topology.

Let us see that the functors are limit preserving.
Take an inductive family $\{A_i\}$ of noetherian rings over $S$.
Take a system $T_i:=Spec(A_i)$, $T=\varprojlim T_i$, $f_i:X_i\frd T_i$, $u_i:T\frd T_i$ the projection morphism, $X=\varprojlim X_i$ and analogously for the log structures.
Using \cite[4, Exp. VII, 5.11]{sga} we see that the morphisms
$$
\varinjlim u_i^{*}R^nf_{i,*}\overline{M}_{X_i}^{gp}(S_i)\frd R^nf_{\infty,*}\overline{M}_{X_{\infty}}^{gp}(S)
$$
are bijective for all $n\geq 0$ in particular condition $[1^{\prime}]$ is also satisfied.

The sheaf $\overline{M}_{X}^{gp}$ is a constructible sheaf of $\zz$-modules by \cite[3.5 ii]{ollog}.  Since the morphism $f$ is proper and special  if we take the base change to an affine scheme $T$, which is the spectrum of a complete local ring with separably closed residue field, we have a decomposition
$$
H^0(T,\overline{M}_{T})\cong \bigoplus_{c_i\in Conn(X_0)}\nn n_{c_i}
$$
and 
$$
\overline{M}_{X_{T}}^{gp}\cong \bigoplus_{c\in Conn(X_0)}\overline{M}_{c}^{gp} 
$$
where $Conn(X_0)$ denotes the set of connected components of the singular locus of the special fiber $X_0$, and $\overline{M}_{c}$ are ``the branches at $c$'' defined as follows
$$
\overline{M}_{c}:=\left\{
\begin{aligned}
x\in \overline{M}_{X_0} \mbox{ such that \'etale locally}\\
\mbox{exists }y\in \overline{M}_{X_0} \mbox{ with } x+y\in (n_c)
\end{aligned}
\right\}
$$
Write $T=Spec(A)$ and let $t_c\in A$ the function corresponding to the component $c\in Conn(X_A)$. The sheaf $\overline{M}_{c}^{gp}$ is supported on $t_c=0$ so we can assume that in $A$ all $t_c$ are zero.
Let $Z_c$ be the corresponding connected component of the locus where $f$ is not smooth.
If $z\in Z_c$ then, \'etale locally, around $z$ the scheme $X_A$ is isomorphic to 
$$
Spec\big(A[x_1,\dots,x_n]/(x_1\cdots x_r)\big)
$$
for some $r\leq n$ and $z$ corresponds to the origin. Let $J_c$ be the ideal corresponding to $Z_c$, i.e. locally given by 
$$(x_1\cdots \hat{x_i}\cdots x_r)_{i=1}^r$$
as in \ref{jeic} and we define 
$$
\nu_c:\tilde{X}_{A,c}\frd X_A
$$
as the proper transform of the blow-up of $X_A$ at $J_c$.
On $X_A$ the sheaf $\overline{M}_c^{gp}$ is isomorphic to $\nu_{c,*}\zz$.\\
The same argument repeats whenever we base change with an affine artinian thickening over $S$ 
$$
Spec(A)\frd S
$$
In particular, given a surjective morphism $A\frd A_0$ of artinian $S$-algebras, with square zero kernel, we have that the map
\begin{equation}\label{fdefzero}
H^0(X_{A},\overline{M}_{X_A}^{gp})\frd H^0(X_{A_0},\overline{M}_{X_{A_0}}^{gp})
\end{equation}
is an isomorphism. Indeed by the previous decomposition over complete local rings we have that the group
$$
H^0(X_{A_0},\overline{M}_{X_{A_0}}^{gp})
$$
is isomorphic to direct sum of copies of $\zz$ indexed by $Conn(X_0)$.
Since for proper morphisms the set of connected components does not changes under nilpotent thickenings (\cite[4, XII, 5.7]{sga}) we obtain the claim. This also shows that the sheaf associated to $f_{*}\overline{M}_X^{gp}$ is formally \'etale over  $S$.

Furthermore the previous argument implies condition $[2^{\prime}]$ for
$$
f_{*}\overline{M}_X^{gp}
$$

Let $A$ be a complete, local noetherian ring with maximal ideal $\mathfrak{m}$ over $S$. Let $A_n:=A/\mathfrak{m}^n$. 
We have proper transforms
$$
\nu_{c,n}:\tilde{X}_{A_n,c}\frd X_{A_n}
$$
First observe that since $\nu_{c,n}$ is finite then $H^1(X_{A_n},\nu_{c,n,*}\zz)=H^1( \tilde{X}_{A_n,c},\zz)$.
Furthermore we have 
\begin{equation}\label{equah1}
H^1(\tilde{X}_{A_n,c},\zz)=H^1(\tilde{X}_{A_0,c},\mathbb{Z})
\end{equation}
Indeed since the morphism $i_n:\tilde{X}_{A_0,c}\rightarrow \tilde{X}_{A_n,c}$ is purely inseparable we have that 
$$
H^1(\tilde{X}_{A_0,c},i_n^{*}\zz)=H^1(\tilde{X}_{A_n,c},\zz)
$$
and since $i_n^{*}\zz=\zz$ we have \ref{equah1}.
In particular using the decomposition $\overline{M}_{X_{A_n}}=\bigoplus \nu_{c,n,*}\zz$ we obtain
$$
H^1(X_{A_n},\overline{M}_{X_{A_n}}^{gp})= H^1(X_{A_0},\overline{M}_{X_{A_0}}^{gp})
$$
For $A$ complete local ring we need to argue a bit differently because the morphism $i_n:\tilde{X}_{A_0,c}\frd \tilde{X}_{A,c}$ is not purely inseparable. Set $Y=\tilde{X}_{A,c}$ (resp. $Y_n=\tilde{X}_{A_n,c}$) and let $\mu:\tilde{Y}\frd Y$ (resp. $\mu_n:\tilde{Y}_n\frd Y_n$) be the strict transform of the blow up at all connected components of the singular locus. 
By \cite[4, IX, 3.6]{sga} $H^1(\tilde{Y},\zz)$ (resp. $H^1(\tilde{Y}_n,\zz)$) are trivial. Consider the cokernel $Q$ (resp. $Q_n$) defined by the sequence
$$
0\frd \zz\frd \mu_{*}\zz\frd Q\frd 0
$$
(resp.
$$
0\frd \zz\frd \mu_{n,*}\zz\frd Q_n\frd 0
$$
)
We have then a commutative diagram
$$
\xymatrix{
    {0}\ar[r] & {H^0(Y,\zz)}\ar[r]\ar[d] & {H^0(\tilde{Y},\zz)}\ar[r]\ar[d]&{H^0(Y,Q)}\ar[r]\ar[d]&{H^1(Y,\zz)}\ar[r]\ar[d] & {0}\\
{0}\ar[r] & {H^0(Y_n,\zz)}\ar[r] &{H^0(\tilde{Y}_n,\zz)}\ar[r]&{H^0(Y_n,Q_n)}\ar[r]&{H^1(Y_n,\zz)}\ar[r] & {0}\\
}
$$
The vertical arrows on the $H^0$ are isomorphisms because we get direct sums of the group $\zz$ indexed by sets which depend only on connected components and on connected components of the singular locus of $Y_0$. This implies that $H^1(Y,\zz)\cong H^1(Y_n,\zz)$.
Summarizing we obtain
\begin{equation}\label{completionr1}
\varprojlim H^1(X_{A_n},\overline{M}_{X_{A_n}}^{gp})=H^1(X_{A},\overline{M}_{X_{A}}^{gp})
\end{equation}
and condition $[2^{\prime}]$ is satisfied by
$$
R^1f_{*}\overline{M}_X^{gp}
$$
The sheaf $f_{*}\overline{M}_X^{gp}$ also satisfies condition $[2^{\prime}]$. The proof is similar and we omit it.
The same argument works if we replace the morphisms $A_{n+1}\frd A_n$ by artinian thickenings $B\rightarrow A$. 
This tells us that the sheaf associated to $R^1f_{*}\overline{M}_X^{gp}$ is formally \'etale over $S$.

We are now going to check $[3^{\prime}]$. Let $A_0$ be local henselian over $S$. Assume we have two sections $\xi_1$ and $\xi_2$ of $f_{*}\overline{M}_X^{gp}(Spec(A_0))$ coinciding on a point $x\in Spec(A_0)$.
Since we are in a group functor we can consider the difference
$$
s=\xi_1-\xi_2\in H^0(X_{A_0},\overline{M}_{X_{A_0}}^{gp})
$$
Now the set of points in $X_{A_0}$ where $s$ is the trivial element in the sheaf $\overline{M}_{X_{A_0},x}^{gp}$ is a non empty open $W\subset X_{A_0}$ because by \cite[3.5]{ollog} it is constructible and stable under generalization. 
Let $Z=X_{A_0}\setminus W$.
By properness of $f$ we have that $f(Z)$ is closed. Its complementary $V$ is then open and non empty because it contains $x$. Over the open $V$ the element $s$ is zero and the two sections coincide.
Since we are looking for open subschemes we can always base change to the henselianization at points and apply the previous argument there to check that the locus where two sections coincide is open.\\
Consider now $[3]^{\prime}(a)$. For the sheaf $f_{*}\overline{M}_X^{gp}$ the previous argument applies.
Let us consider $R^1f_{*}\overline{M}_X^{gp}$. Since the statement is true at the closed point $a\in Spec(A_0)$ and $A_0$ is a discrete valuation ring it is enough to find an open for which the statement is true. 
It is enough to check this over the strict henselianization $A^h$ of $A_0$ at the closed point $a$ of $Spec(A_0)$. 
Using the proper base change theorem in the form of \cite[4, Exp. XII, 5.5]{sga} we have an inclusion
$$
R^1f_{*}\overline{M}_X^{gp}(A^h)=H^1(X_{A^h}, \overline{M}_{X_{A^h}}^{gp})\hookrightarrow H^1(X_a, \overline{M}_{X_a}^{gp})
$$
\noindent hence the two torsors coincide globally.

Let us check $[3]^{\prime}(b)$. 
For the sheaf $f_{*}\overline{M}_X^{gp}$ there are no problems by the same argument as before.
Let us consider $R^1f_{*}\overline{M}_X^{gp}$. Again it is enough to check the statement after base change to the henselianization. If $x$ denotes one of the points in the dense set and $A_x$ denotes the strict henselianization of $A_0$ at $x$ it is enough to find an open $x\in V\subset Spec(A_x)$ where the statement is true.
We apply as before the proper base change theorem and we get $V=Spec(A_x)$.

From conditions \ref{fdefzero} and formal \'etaleness of $R^1f_{*}\overline{M}_X^{gp}$ we get that the module of the deformation theory is the zero module for both functors. Furthermore obstruction theory is also trivial. In particular conditions $[4^{\prime}]$ and $[5^{\prime}]$ are also satisfied by both functors.
The absence of obstructions to lift over infinitesimal thickenings and trivial deformation theory tells us that the corresponding algebraic spaces are \'etale over $S$.

\end{proof}

Let us see $(1)\Rightarrow (2)$. Since we assume that $K$ is \'etale over $S$ then the quotient $\pic_{C/S}/K$ is representable by an algebraic space over $S$.
Using the previous lemma we have a sequence 
\begin{equation}\label{mezzacorta}
0\frd \pic_{C/S}/K\frd \piclog_{C/S}\frd R^1f_{*}\overline{M}_{C}^{gp}
\end{equation}
with left and right member being algebraic spaces in groups.
Since $R^1f_{*}\overline{M}_{C}^{gp}$ is \'etale over $S$ it is not difficult to see that the connected component of the identity $R^1f_{*}\overline{M}_C^{gp,0}=0_S$.
This implies that we have isomorphisms of sheaves
\begin{equation}\label{pictaueta}
    \piclogtau_{C/S}=\piclog_{C/S}\times_{R^1 f_{*}\overline{M}_C^{gp}} R^1f_{*}\overline{M}_C^{gp,0}=\pic_{C/S}/K
\end{equation}
Hence $\piclogtau_{C/S}$ is an algebraic space in groups. 
The open immersion $R^1f_{*}\overline{M}_C^{gp,0}\frd R^1f_{*}\overline{M}_C^{gp}$ is \'etale hence from \ref{pictaueta} we have that $\piclogtau_{C/S}\frd \piclog_{C/S}$ is formally \'etale.

Given an artinian thickening $A\frd A_0$ of rings over $S$ with square zero kernel $I$, we have the log exponential sequence (\cite[4.12.1]{olpic})
\begin{equation}\label{expseq}
0\frd \ox_{C_{A_0}}\otimes I\frd M_{C_A}^{gp}\frd M_{C_{A_0}}^{gp}\frd 0
\end{equation}

Using the log and the classical exponential sequence and the fact that 
$$H^2(C_s,\ox_{C_s})=0$$
over fibers $s\in S$ we see that $\pic_{C/S}$, $\piclog_{C/S}$ and $\piclogtau_{C/S}$ are formally smooth over $S$.
Given an artinian thickening
$$
A\frd A_0
$$
over $S$ as before and a sheaf $G$ on $(Sch/S)_{et}$ we have a canonical map
$$
G(Spec(A))\frd G(Spec(A_0))
$$
Assume now that $G$ is a sheaf in groups and define a functor over artinian thickenings with square zero kernel 
$$
A\frd A_0
$$
 over $S$ as
$$
\mathcal{L}ie(G)(A\frd A_0):= \ker(G(Spec(A))\frd G(Spec(A_0))
$$
Given a morphism of group sheaves $f:F\frd G$ we obtain a morphism
$$
\mathcal{L}ie(F)(A\frd A_0)\frd \mathcal{L}ie(G)(A\frd A_0)
$$
\begin{lm}\label{lemmagrfun}
        Let $f:F\rightarrow G$ be a formally \'etale morphism of group sheaves $S$ then for any artinian thickening $A\frd A_0$ over $S$ with square zero kernel the induced map
            $$
            \mathcal{L}ie(F)(A\frd A_0)\frd \mathcal{L}ie(G)(A\frd A_0)
            $$
is an isomorphism.
\end{lm}
\begin{proof}
    This is an easy exercise.

\end{proof}
Since $\piclogtau_{C/S} \frd \piclog_{C/S}$ is formally \'etale we have by the previous lemma
$$
\mathcal{L}ie(\piclogtau_{C/S} )=\mathcal{L}ie(\piclog_{C/S})
$$

The same lemma and the fact that we are assuming that $K$ is an \'etale group over $S$ imply
$$
\mathcal{L}ie(\pic_{C/S})=\mathcal{L}ie(\pic_{C/S}/K)
$$
Combining these facts we obtain for any artinian thickening $A\frd A_0$ over $S$ with square zero kernel that the canonical morphism 
\begin{equation}\label{isotangent}
\mathcal{L}ie(\pic_{C/S})(A\frd A_0) \frd \mathcal{L}ie(\piclog_{C/S})(A\frd A_0)
\end{equation}
induced by the sequence \ref{mezzacorta} is an isomorphism.  
Using this sequence and the classical exponential sequence we can read the morphism 
$$\mathcal{L}ie(\pic_{C/S})(A\rightarrow A_0)\frd \mathcal{L}ie( \piclog_{C/S})(A\rightarrow A_0)$$
from the diagram
\begin{equation}\label{tangentdiagram}
\xymatrix{
{H^0(C_{A_0},\gm_{C_{A_0}})}\ar[d]_{a} \ar[r]&  {H^0(C_{A_0},M_{C_{A_0}}^{gp})}\ar[d]^{b}\\
{H^1(C_{A_0},\ox_{C_{A_0}}\otimes I)} \ar[d]\ar[r] &{H^1(C_{A_0},\ox_{C_{A_0}}\otimes I)}\ar[d]\\
 {H^1(C_{A},\gm_{C_{A}})}\ar[d]\ar[r] & {H^1(C_{A},M_{C_{A}}^{gp})}\ar[d] \\
 {H^1(C_{A_0},\gm_{C_{A_0}})} \ar[r]&  {H^1(C_{A_0},M_{C_{A_0}}^{gp})}\\
}
\end{equation}
Namely we have
$$
\mathcal{L}ie( \piclog_{C/S})(A\frd A_0)=H^1(C_{A_0},\ox_{C_{A_0}}\otimes I)/im(b)
$$
and 

$$
\mathcal{L}ie( \pic_{C/S})(A\frd A_0)=H^1(C_{A_0},\ox_{C_{A_0}}\otimes I)/im(a)
$$
In particular isomorphism \ref{isotangent} gives an isomorphism between the module 
$$
H^1(C_{A_0},\ox_{C_{A_0}}\otimes I)/im(a)
$$
 and the module 
$$
H^1(C_{A_0},\ox_{C_{A_0}}\otimes I)/im(b)
$$
 Remember now that by \cite[III, 7.8.6]{ega} or \cite[8.1.8]{blr} the functor 
$$
T\frd \Gamma(C_T,\ox_{C_T})
$$
is represented by a vector bundle $V$ over $S$ if and only if $f:C\frd S$ is cohomologically flat in dimension zero. In this case the subfunctor 
$$
\Gamma(\ox_C^{*}):T\frd \Gamma(C_T,\ox_{C_T}^{*})
$$
is represented by an open subgroup scheme (\cite[8.2.10]{blr}). In particular since our family is (classically) cohomologically flat in dimension zero, the functor $\Gamma(\ox_{C}^{*})$ is smooth because it is an open of a smooth functor. As consequence the morphism $a$ in diagram \ref{tangentdiagram} has to be zero because smoothness implies the absence of obstructions to lift over infinitesimal thickenings. In this way we get an isomorphism 
$$
H^1(C_{A_0},\ox_{C_{A_0}}\otimes I)\stackrel{\cong}{\frd} H^1(C_{A_0},\ox_{C_{A_0}}\otimes I)/im(b)
$$
But this implies that $b$ is the zero map, in particular from the sequence
$$
H^0(C_{A},M_{C_A}^{gp})\frd H^0(C_{A_0},M_{C_{A_0}}^{gp})\stackrel{b}{\frd} H^1(C_{A_0},\ox_{C_{A_0}}\otimes I)
$$
 we obtain log cohomological flatness. 
 
Let us prove that $(2)\Rightarrow (1)$.
Observe that $(2)$ together with theorem \ref{olssonrepres} implies that $\piclogtau_{C/S}$ is an algebraic space over $S$. 
Since kernels of morphisms of algebraic spaces in groups are representable then also $K$ is representable by an algebraic space over $S$.
The morphism induced by $\delta$ 
\begin{equation}\label{appofla}
f_{*}\overline{M}_{C/S}^{gp} \frd K\frd 0
\end{equation}
is unramified. Indeed we have a surjection
$$
\Omega_{f_{*}\overline{M}_{C/S}^{gp}/S}^1\frd\Omega_{f_{*}\overline{M}_{C/S}^{gp}/K}^1\frd 0
$$
and $\Omega_{f_{*}\overline{M}_{C/S}^{gp}/S}^1=0$ by \'etaleness of $f_{*}\overline{M}_{C/S}^{gp}$ over $S$. If we can show that the morphism in \ref{appofla} is flat then $K\frd S$ would be \'etale because \'etaleness descends from flat covers. 

Consider the group algebraic space $K^{\prime}=\ker(\delta)$. Since we have an exact sequence
$$
0\frd K'\frd f_{*}\overline{M}_{C/S}^{gp}\frd K\frd 0
$$
and the morphism \ref{appofla} is unramified then also $K'\frd S$ is unramified. To show that the morphism \ref{appofla} is flat it is enough to show that $K'\frd S$ is flat. 
We have a surjection of \'etale sheaves 
\begin{equation}\label{surjettop}
f_{*}M_{C/S}^{gp}\frd K'\frd 0
\end{equation}
Observe that $f_{*}M_{C/S}^{gp}$ is representable by an algebraic space over $S$ because $f_{*}\mathbb{G}_m$ and $K'$ are representable and $f_{*}\gm$ is smooth over $S$ (\cite[7.3]{aim}).
By log cohomological flatness we have that $f_{*}M_{C/S}^{gp}$ is formally smooth over $S$.
Since $K'\frd S$ is formally unramified and the morphism in \ref{surjettop} is surjective for the \'etale topology it is easy to see that $K'\frd S$ is also formally \'etale. But then $K'$ is \'etale, hence flat over $S$.
This proves the implication $(2)\Rightarrow (1)$ in theorem \ref{generalrepres}.
\end{proof}
\subsection{The case of aligned curves}
We show now a nice consequence of the alignment condition.
\begin{thm}\label{teorema}
 Let $(S,M_S)$ be a log scheme such that $S$ is regular. Assume that the log structure $M_S$ is trivial on a schematically dense open $U\subset S$. Let 
$$
f:(C,M_C)\frd (S,M_S)
$$
\noindent be a special, aligned and log semistable curve over $(S,M_S)$, strict over $U\subset S$. 
Then the morphism $f:(C,M_C)\frd (S,M_S)$ is log cohomologically flat and the sheaf $\piclogtau_{C/S}$ is representable by an algebraic space smooth over $S$.
\end{thm}
\begin{proof}
Once we have representability the smoothness statement has already been discussed.
    Consider the sheaf $K=im(f_{*}\overline{M}_C^{gp}\frd \pic_{C/S})=\ker(\pic_{C/S}\frd \piclog_{C/S})$. Since $f_{*}\overline{M}_C^{gp}$ and $\pic_{C/S}$ are representable then $K$ is an algebraic space over $S$. 
\begin{lm}\label{imagedelta}
\noindent Under the hypothesis of the theorem the sheaf $K$ is an algebraic space \'etale over $S$. 
\end{lm}
\begin{proof}
The space $f_{*}\overline{M}_{C}^{gp}$ is \'etale over $S$ by lemma \ref{mbaralgsp} and in particular $f_{*}\overline{M}_C^{gp}$ is reduced.
From this follows that the image of the morphism 
$$
\delta:f_{*}\overline{M}_C^{gp}\frd \pic_{C/S}
$$
is also reduced. This morphism is also unramified by \'etaleness of $f_{*}\overline{M}_C^{gp}$ over $S$.
By the alignment hypothesis we have that $E$ is \'etale over $S$ (theorem \ref{hol:main}).
Over the open $U\subset S$ the log structure is trivial hence there is no difference between $M_{C_U}^{gp}$ and $\mathbb{G}_{m,C_U}$. 
In this way we get isomorphisms
 $$
 f_{*}\overline{M}_{C}^{gp}|_U\stackrel{\cong}{\frd}K|_U \stackrel{\cong}{\frd}E|_U\cong e_U
 $$
and 
$$
\pic_{C/S}|_U\cong \piclog_{C/S}|_U
$$
It is enough to prove that we have a factorization
$$
\xymatrix{
    {f_{*}\overline{M}_{C/S}^{gp}}\ar[rr]^{\delta}\ar@{.>}[dr] && {\pic_{C/S}}\\
    & {E }\ar[ur]
}
$$
for the boundary morphism $\delta$. 
Indeed assuming this then by alignment $E$ is \'etale over $S$ but then the morphism $$
f_{*}\overline{M}_{C/S}^{gp}\rightarrow E
$$
is also \'etale. 
Since \'etaleness descends from flat coverings we must have that also $K$ is \'etale over $S$.
To show the factorization it is enough to show that for any $s\in S$ we have such factorization.
Since algebraic spaces in groups over fields are representable by group schemes we have that $f_{*}\overline{M}_{C_s}^{gp}$ and $\pic_{C_s/s}$ are group schemes locally of finite type over $s$. 
We can even assume that $s$ is a geometric point. Then by very density of $s$-points and the fact that $K$ and $E$ are reduced it is enough to prove that set theoretically we have 

\begin{equation}\label{equalityclosedpoints}
    \delta_s( f_{*}\overline{M}_{C_s}^{gp})(s)\subset E_s(s)
\end{equation}

We have already seen that this is true at fibers over points $s\in U$. Take a geometric point $s\in S\setminus U$.

Let $V=Spec(R)$ with $R$ a strictly henselian discrete valuation ring with algebraically closed residue field (isomorphic to the residue field of $s$) and  $V\frd S$ be a morphism mapping the closed point $v\in V$ to $s$ and the generic point $\eta$ to $U$. 
Since $E$ is \'etale then also the base change $E_V\frd V$ is \'etale. 
The morphism $E\frd \pic_{C/S}$ is a closed immersion hence the base change morphism $E_V\frd \pic_{C_V/V}$ is a closed immersion too.
Since the preimage of $U$ in $V$ is schematically dense in $V$ and $E_V$ is flat over $V$ we have that the preimage of the generic point $\eta$ is schematically dense in $E_V$ . In particular $E_V$ is isomorphic the closure of the unit section of $\pic_{C_\eta/\eta}$ in $\pic_{C_V}$. 

Let now $\mathcal{D}$ be the group of Cartier divisors on $C_{V}$ whose support is concentrated in the special fiber of $C_{V}$ and $\mathcal{D}_0$ be the subgroup of principal divisors.

By \cite[6.1.3]{ray} and \cite[6.4.1]{ray} over discrete valuation rings there are identifications
\begin{equation}\label{closunitdvr}
    E_{V}(v)= E_{V}(V)= \mathcal{D}/\mathcal{D}_0
\end{equation}
At this point the inclusion \ref{equalityclosedpoints} follows by an easy calculation with \v{C}ech cohomology or by the characterization of the morphism $\delta$ given in \cite[3.2, 3.3]{olpic}.
Let us review this. We consider two different log structures: the special on $(C,M_C)\frd (V,M_V)$ and the divisorial one $(C,M_C^{\dag})/(V,M_V^{\dag})$ induced from the special fiber.
Similar as in \cite[3.2, 3.3]{olpic} and in section 1.1 of this paper we have a commutative diagram
$$
\xymatrix{
    {0}\ar[r] & {\ox_C^{\times}}\ar[r]\ar[d]^{=} & {M_C^{gp}}\ar[r]\ar[d] & {\oplus \nu_{c,*}\zz}\ar[r]\ar[d] & {0}\\
    {0}\ar[r] & {\ox_C^{\times}}\ar[r] & {M_C^{\dag, gp}}\ar[r] & {\nu_{*}\zz}\ar[r] & {0}\\
}
$$
where $\nu_c:\tilde{C}_{c,v}\frd C_v$ is the partial normalization of the special fiber $C_v$ at the node $c$ and $\nu:\tilde{C}_v\frd C_v$ is the normalization of the special fiber.
Denote with $\underline{\mbox{Pic}}_{C_V/V}^{\dag,log}$ the log Picard functor constructed using the $M_{C_V}^{\dag}$ log structure.
The previous diagram gives us 

\begin{equation}\label{grandediagramma}
\xymatrix{
    &        &                              & {\piclogtau_{C_V/V}}\ar[d]^{i} & &\\ 
    \dots \ar[r]&   {f_{*}\overline{M_{C_V}^{gp}}}\ar[r]^{\delta_V}\ar[d] & {\pic_{C_V/V}}\ar[r]\ar[d]^{=}\ar[ur] & {\piclog_{C_V/V}}\ar[r]\ar[d]^{p} & {R^1f_{*}\overline{M}_{C_V}^{gp}}\ar[r]\ar[d] & \dots \\
    \dots \ar[r]&   {f_{*}\overline{M}_{C_V}^{\dag,gp}}\ar[r]^{\delta^{\dag}} & {\pic_{C_V/V}}\ar[r] & {\underline{\mbox{Pic}}_{C_V/V}^{\dag,log}}\ar[r] & {R^1f_{*}\nu_{*}\zz \stackrel{(*)}{=}0}  &
}
\end{equation}
Where (*) follows from \cite[4, IX 3.6]{sga}.
The morphism $\delta_V$ has then the following factorization

\begin{equation}\label{deltaolsson}
    \xymatrix{
        {f_{*}\overline{M}_{C_V}^{gp}}\ar[d]\ar[rr]^{\delta_V} && {\pic_{C_V/V}}\ar[d]^{\cong}\\
        {\mathbb{Z}^{\sharp\{\mbox{connected components of }\tilde{C}_v \}}}\ar@{.>>}[dr]\ar[rr]^{\delta^{\dag}}&& {\pic_{C_V/V}}\\
        &{ \mathcal{D}/\mathcal{D}_0=E_V(v)}\ar[ur] &
    }
\end{equation}
where the middle horizontal morphism sends an irreducible component $Z$ of $\tilde{C}_v$ to a line bundle $\mathcal{L}_{Z}$ on $C_v$ having the following properties. Let $\bar{Z}$ be the image of $Z$ in $C_v$. For every irreducible component $Y$ of $C_v$ different from $\bar{Z}$ we have
$$
\mathcal{L}_{Z}|_{Y}\cong \ox_{Y}(-D_{Y,Z})
$$
\noindent where $D_{Y,Z}=\bar{Z}\cap Y$ and
$$
\mathcal{L}_{Z}|_{\bar{Z}}=\ox_{\bar{Z}}(\sum_{Y\neq \bar{Z}} D_{Y,Z})
$$
This together with \ref{closunitdvr} shows inclusion \ref{equalityclosedpoints} and the proof of the lemma is complete.

\end{proof}

The lemma and theorem \ref{generalrepres} complete the proof of \ref{teorema}. 
\end{proof}

The proof of lemma \ref{imagedelta} shows that under the hypothesis of theorem \ref{teorema} the whole sheaf $\piclog_{C/S}$ is slightly more separated than $\pic_{C/S}$.
Suppose that the curve is aligned so that $E$ is \'etale over $S$. By log cohomological flatness we have that the morphism
$$
\pic_{C/S}\frd\piclogtau_{C/S}
$$
is \'etale. 
The quotient $F=E/K$ is then representable by an algebraic space and this is a subgroup of $\piclogtau_{C/S}$ \'etale over $S$.
In particular the sheaf $\piclogtau_{C/S}/F$ is representable by an algebraic space which is smooth over $S$.
Since we have already seen that $\piclogtau_{C/S}=\pic_{C/S}/K$ we have
$$
\piclogtau_{C/S}/F=(\pic_{C/S}/K)/(E/K)=\pic_{C/S}/E
$$
This implies that $\piclogtau_{C/S}/F$ is a separated algebraic space over $S$. 
From this follows that $F$ is a closed subgroup of $\piclog_{C/S}$ and it is \'etale over $S$. Let $t:F\frd S$ be the structure morphism. Since $U$ is schematically dense in $S$ and $F/S$ is \'etale we have that $t^{-1}U=e_U$ is schematically dense in $F$. It follows that $F$ is the closure in $\piclogtau_{C/S}$ of the unit section $e_U\in \piclogtau_{C_U/U}=\pic_{C_U/U}$.
This implies that 
$$
\piclogtau_{C/S}/F
$$
is the maximal separated quotient of $\piclogtau_{C/S}$.
Because of this the next definition is well posed. 
\begin{definition}\label{maxsepquot}
    Let $f:(C,M_C)\frd (S,M_S)$ be a family of curves as in theorem \ref{teorema}. 
    \begin{enumerate}
        \item Define $Q^{log}$ as the maximal separated quotient of $\piclogtau_{C/S}$.
        \item Suppose that $S$ is the spectrum of a discrete valuation ring. Let $C/S$ be a proper family of nodal curves with smooth generic fiber. Define 
            $$
            f^{\dag}:(C,M_C^{\dag})\frd (S,M_S^{\dag})
            $$
            be the morphism of log schemes such that the log structure is the divisorial one induced by the special fiber. In this situation define $\piclogdag_{C/S}$ to be the \'etale sheaf whose section over a scheme $T/S$ are given by
            $$
            \piclogdag_{C/S}(T)=H_{et}^0(T,R^1f_{T,*}^{\dag}M_{C_T}^{\dag, gp})
            $$
    \end{enumerate}
\end{definition}
\begin{thm}\label{teoremamain}
Let $f:(C,M_C)\frd (S,M_S)$ as in theorem \ref{teorema}. Then
\begin{enumerate}
    \item the maximal separated quotient $Q^{log}$ of $\piclogtau_{C/S}$ exists and if $C$ is regular then it is a N\'eron model for $\pic_{C_U/U}$;
    \item the algebraic space $Q^{log}$ has the property that for any discrete valuation ring $V$ and morphism $T=Spec(V)\frd S$ mapping the generic point to the open $U$ there is a canonical isomorphism of group functors

$$
Q_T^{log}\cong \piclogdag_{C_T/T}
$$
where $Q_T^{log}=Q^{log}\times_S T$.
\end{enumerate}

\end{thm}

\begin{proof}
    Existence in part (1) has already been discussed before definition \ref{maxsepquot}. Corollary \ref{Eetale} implies that if $C$ is regular then $Q^{log}$ is a N\'eron model of $\pic_{C_U/U}$.\\
    Part (2) follows from diagram \ref{deltaolsson} and from the surjectivity of the composition $p\circ i$ in diagram \ref{grandediagramma} once we observe that the formation of $\piclogtau_{C/S}$ commutes with base change and that \'etaleness of $E$ implies that the pull back of $E$ to $T$ is isomorphic to the closure of the unit section of $\pic_{C_{\eta}/\eta}$ in $\pic_{C_T/T}$.
\end{proof}

\begin{cor}\label{functnerodvr}
Given a proper nodal curve over the spectrum of a discrete valuation ring $T$ with smooth generic fiber $C_{\eta}$ then the sheaf $\piclogdag_{C/T}$ is representable by an algebraic group which is smooth and separated over $T$. If $C$ is regular then $\piclogdag_{C/T}$ is isomorphic to the classical N\'eron model of $\pic_{C_{\eta}/\eta}$.
\end{cor}
\begin{proof}

A curve as in the statement is automatically aligned. The corollary follows then from part (1) and (2) of the previous theorem and from \cite{anan} th\'eor\`eme 4.B.
\end{proof}
We give now a partial converse of theorem \ref{teorema} by showing that a priori alignment is a stronger condition than log cohomological flatness.
\begin{cor}
    Let $(S,M_S)$ be a log scheme with $S$ regular. Let $f:(C,M_C)\frd (S,M_S)$ be a special, log semistable curve. Suppose there exists a schematically open dense $U\subset S$ such that the morphism is strict over $U$ and $M_U$ is trivial. Suppose that the family is log cohomologically flat, that $Q^{log}$ exists and that the quotient morphism $q:\piclogtau_{C/S}\frd Q^{log}$ is flat then the family of curves $C\frd S$ is aligned and $Q^{log}$ is smooth over $S$. 
\end{cor}
\begin{proof}
By log cohomological flatness $\piclogtau_{C/S}$ is a smooth algebraic space over $S$ and the morphism $\pic_{C/S}\frd \piclogtau_{C/S}$ is \'etale.
Consider the composition $p:\pic_{C/S}\frd \piclogtau_{C/S}\frd Q^{log}$. 
Since $q$ is flat by assumption this implies that $p$ is flat.
Let $F=\ker(p)$. 
This is an algebraic space flat over $S$ and the inclusion $F\frd \pic_{C/S}$ is a closed immersion because $Q^{log}$ is separated over $S$. Since $F|_U$ is the unit section we must have 
$$
E\subset F
$$
Let $t:F\frd S$ be the structure morphism. Since this is flat then the open $E|_U=e_U=F|_U=t^{-1}U$ is schematically dense in $F$. This implies that $E=F$.
Since now $E$ is flat over $S$ we can use theorem \ref{hol:main} to conclude that the curve $C/S$ is aligned and that $F$ is \'etale over $S$. 
Since $\pic_{C/S}$ is smooth and $Q^{log}=\pic_{C/S}/F$ then $Q^{log}$ is smooth over $S$.
\end{proof}
We conclude by giving an easy consequence of theorem \ref{teorema}.
\begin{cor}
Let $(S,M_S)$ be a log scheme such that $S$ is excellent and regular in codimension 1 and that the log structure is trivial on a schematically dense, regular open subscheme $U\subset S$. 
Let $$f:(C,M_C)\frd (S,M_S)$$ be a special, log semistable curve, whose underlying scheme map is smooth over $U$, and such that the log structure on $C$ is trivial on $f^{-1}U$. There exists an open $U\subset V\subset S$, whose complementary is of codimension at least 2 in $S$, such that
$$
\piclogtau_{C_V/V}
$$
\noindent is representable by a smooth algebraic space in groups.
\end{cor}
\begin{proof}
Since $S$ is excellent and it has a non empty open regular subscheme, the regular locus is a non empty open in $S$ and it has complementary of codimension at least 2 (by regularity in codimension one).
In particular after shrinking we can assume that $S$ is regular. By the regularity hypothesis on $U$ we still have $U\subset S$.
Furthermore outside codimension at least 2 the alignment condition is always satisfied in our situation. Hence shrinking again we can assume that $C\frd S$ is aligned.
Using theorem \ref{teorema} we have that $\piclogtau_{C/S}$ is representable.
\end{proof}

\end{document}